\documentclass[10pt]{article}

\usepackage{latexsym}
\usepackage{amsbsy}
\usepackage{amssymb}
\usepackage{amsfonts}
\usepackage{amsmath,amsthm}
\usepackage[latin1]{inputenc}%needed for bibliography
\usepackage{a4wide}
 \newtheorem{thm}{Theorem}[section]
 \newtheorem{defin}[thm]{Definition}
 \newtheorem{lem}[thm]{Lemma}
 \newtheorem{prop}[thm]{Proposition}
 \newtheorem{cor}[thm]{Corollary}
 \theoremstyle{definition} % added by gue to make text in rem and ex non-italic
 \newtheorem{rem}[thm]{Remark}
 \newtheorem{ex}[thm]{Example}

 \newcommand{\bthm}{\begin{thm}}
 \newcommand{\ethm}{\end{thm}}
 \newcommand{\bd}{\begin{defin}}
 \newcommand{\ed}{\end{defin}}
 \newcommand{\blem}{\begin{lem}}
 \newcommand{\elem}{\end{lem}}
 \newcommand{\bcor}{\begin{cor}}
 \newcommand{\ecor}{\end{cor}}
 \newcommand{\bprop}{\begin{prop}}
 \newcommand{\eprop}{\end{prop}}
 \newcommand{\brem}{\begin{rem} \rm}
 \newcommand{\erem}{\end{rem}}
 \newcommand{\bex}{\begin{ex} \rm}
 \newcommand{\eex}{\end{ex}}

 \newcommand{\pr}{\noindent{\bf Proof. }}
 \newcommand{\ep}{\nolinebreak{\hspace*{\fill}$\Box$ \vspace*{0.25cm}}}

 \newcommand{\beq}{\begin{equation}}
 \newcommand{\eeq}{\end{equation} }

 \newcommand{\bea}{\begin{eqnarray}}
 \newcommand{\eea}{\end{eqnarray}}

 \newcommand{\beas}{\begin{eqnarray*}}
 \newcommand{\eeas}{\end{eqnarray*}}

 \newcommand{\beqs}{\begin{equation*}}
 \newcommand{\eeqs}{\end{equation*}}

 \newcommand{\bi}{\begin{itemize}}
 \newcommand{\ei}{\end{itemize}}

 \newcommand{\ben}{\begin{enumerate}}
 \newcommand{\een}{\end{enumerate}}

 \newcommand{\ba}{\begin{array}}
 \newcommand{\ea}{\end{array}}

 \newcommand{\R}{\mathbb R}
 \newcommand{\N}{\mathbb N}
 \newcommand{\C}{\mathbb C}
 \newcommand{\Z}{\mathbb Z}
 \newcommand{\K}{\mathbb K}
 
 \newcommand{\Rt}{\widetilde{\mathbb R}}
 \newcommand{\Ct}{\widetilde{\mathbb C}}
 \newcommand{\Kt}{\widetilde{\mathbb K}}
 \newcommand{\sig}{\sigma}
 \newcommand{\de}{\delta}
 \newcommand{\la}{\lambda}
 \newcommand{\al}{\alpha}
 \newcommand{\be}{\beta}
 \newcommand{\om}{\omega}
 \newcommand{\col}{\colon}
 \newcommand{\dual}[2]{\langle #1, #2 \rangle}

 \newcommand{\Cinf}{\cC^\infty}

 \newcommand{\cB}{\ensuremath{{\cal B}}}
 \newcommand{\cC}{\ensuremath{{\cal C}}}
 
 \newcommand{\cE}{\ensuremath{{\cal E}}}

 \newcommand{\cN}{\ensuremath{{\cal N}}}

 \newcommand{\eps}{\varepsilon}

\newcommand{\Span}{\ensuremath{\mathrm{span\,}}}

\begin{document}

 \title{Symplectic modules over Colombeau-generalized numbers
 %and foundations of non-smooth symplectic geometry
}

 \author{G\"unther H\"ormann
         \footnote{Faculty of Mathematics, University of Vienna,
         Nordbergstr.\ 15, A-1090 Vienna, Austria,
         Electronic mail: guenther.hoermann@univie.ac.at}\\
         Sanja Konjik
         \footnote{Faculty of Sciences, Department of Mathematics and Informatics, University of Novi Sad,
         Trg Dositeja Obradovi\'ca 4, 21000 Novi Sad, Serbia,
         Electronic mail: sanja.konjik@dmi.uns.ac.rs}\\
         Michael Kunzinger
         \footnote{Faculty of Mathematics, University of Vienna,
         Nordbergstr.\ 15, A-1090 Vienna, Austria,
         Electronic mail: michael.kunzinger@univie.ac.at.}
       }

 \date{\today}
 \maketitle

 \begin{abstract}
We study symplectic linear algebra over the ring $\Rt$ of Colombeau generalized numbers. 
Due to the algebraic properties of $\Rt$ it is possible to preserve a number of central results
of classical symplectic linear algebra. In particular, we construct symplectic bases for 
any symplectic form on a free $\Rt$-module of finite rank. Further, we consider the general problem
of eigenvalues for matrices over $\Kt$ ($\K=\R$ or $\C$) and derive normal forms
for Hermitian and skew-symmetric matrices. Our investigations are motivated by 
applications in non-smooth symplectic geometry and the theory of Fourier integral
operators with non-smooth symbols.

 \vskip5pt
 \noindent
 {\bf Mathematics Subject Classification (2010):}
 15A63, 15A18, 15A21, 46F30. 
 \vskip5pt
 \noindent
 {\bf Keywords:} Symplectic linear algebra, symplectic modules, Colombeau generalized numbers, generalized eigenvalues.
 \end{abstract}

%%%%%%%%%%%%%%%%%%%%%%%%%%%%%%%%%%%%%%%%%%%%%%%%%%%%%%%%%%%%%%%%%%
 \section{Introduction} \label{sec:intro}
%%%%%%%%%%%%%%%%%%%%%%%%%%%%%%%%%%%%%%%%%%%%%%%%%%%%%%%%%%%%%%%%%%

Algebras of generalized functions in the sense of Colombeau constitute a 
valuable tool for studying singular problems including nonlinearities and
have found numerous applications in PDEs, non-smooth differential geometry,
and mathematical physics (cf., e.g., \cite{C:84,C85,NPS98,O:92,book}). 
Over the past few years, increased attention has been given to algebraic
properties of the ring of generalized numbers $\Kt$ (for $\K=\R$ or $\C$), cf.\ 
\cite{AFJ09,AJ,AJOS:08,EbM-LGiAGF,Vernaeve:10,BK:12} and the references therein.

In the present paper we study symplectic linear algebra on free $\Rt$-modules
of finite rank. This naturally leads to further fundamental algebraic notions
like eigenvalues and spectral properties, as well as normal forms for certain
types of matrices over generalized numbers. 

Our investigation of symplectic linear structures here is the starting
point -- in terms of tangent space constructions -- for the systematic
development of non-smooth symplectic differential geometry, which lies at
the heart of deeper applications in microlocal analysis and mechanics or general
relativity on semi-Riemannian manifolds. For example, one of the main current
issues in research on propagation of singularities for (pseudo)differential
operators with non-smooth principal symbol on a manifold $M$ is to understand
 the  precise relation with the bicharacteristic flow on the
cotangent bundle $T^* M$, i.e, the flow of the non-smooth Hamiltonian vector field
steming from the principal symbol as Hamilton function (cf.\ \cite{GH:05,O:09}).
Furthermore, the microlocal mapping properties of solution operators are then
governed by generalized Fourier integral operators with non-smooth phase functions
and symbols and the wave front sets of their kernels, which are described in terms
of generalized Lagrangian submanifolds in appropriate cotangent bundles (cf.\
\cite{GHO:09,GO:12}). As for mechanics and general relativity in the context of
non-smooth metrics or space-times, a prominent problem is to study the geodesic
flow on a generalized semi-Riemannian manifold $(M,g)$. This is  the flow on $T
M$ of the non-smooth geodesic spray $G$ which is the vector field
on $TM$ given in coordinates $(x,v)$ on $TM$ by 
$$ 
   G(x,v) =  \sum_{1 \leq j \leq n}  v_j\, \partial_{x_j} - 
   \sum_{1 \leq j,k,l \leq n} \Gamma^j_{k l}\, v_k\, v_l\, \partial_{v_j}, 
$$
where $\Gamma^j_{k l}$ are the Christoffel symbols (cf.\
\cite{AM:78}). Non-degeneracy of the metric gives rise to the `non-smooth
diffeomorphism' $g^\flat \col TM \to T^* M$ and $\sig := (g^\flat)^* \om$ defines
a non-smooth symplectic form on $TM$ via pull-back of the canonical form $\omega$
on $T^* M$. Locally, the non-smooth symplectic form $\sig$ is explicitly given by
$$
  \sig = \sum_{1 \leq i,j \leq n} g_{ij}\, dx_i \wedge dv_j + 
  \sum_{i,j,k} \frac{\partial g_{ij}}{\partial{x_k}} \, v_i\, dx_j \wedge dx_k.
$$
The geodesic flow can then be studied in terms of a family of
symplectomorphisms on the generalized symplectic manifold $(TM, \sig)$.

The paper is organized as follows. To make the presentation reasonably self-contained, in
Section \ref{sec:gf} we provide basic definitions and present a number of
fundamental
algebraic properties in the Colombeau setting. Section \ref{basicsec} then turns
to symplectic forms on $\Rt$-modules. We prove the existence of symplectic bases and
of symplectic basis extensions and study symplectic maps, as well as symplectic submodules
(isotropic, involutive, Lagrangian). In Section \ref{matrixsec} our main focus lies on 
spectral properties of matrices over $\Kt$.    We study an appropriate notion of eigenvalues 
in the present context. We also derive a 
characterization of eigenvalues in terms of determinants. Finally, in Section \ref{skewsec}
we study Hermitian and skew-symmetric matrices. For these specific types of matrices, 
we show that there is always a distinguished set of eigenvalues, based on which 
normal forms can be derived.

%%%%%%%%%%%%%%%%%%%%%%%%%%%%%%%%%%%%%%%%%%%%%%%%%%%%%%%%%%%%%%%%%%
 \section{The ring of generalized numbers} \label{sec:gf}
%%%%%%%%%%%%%%%%%%%%%%%%%%%%%%%%%%%%%%%%%%%%%%%%%%%%%%%%%%%%%%%%%%
In this section we collect a number of fundamental properties of the ring of generalized numbers
$\Kt$, defined (for $\K=\R$ or $\K=\C$) as follows  
\bd For $I:=(0,1]$ set
\[
\begin{split}
\cE_M &:= \{(r_\eps)_{\eps\in I}\in \K^I : \exists N\in \N\ |r_\eps| = O(\eps^{-N}) \ (\eps\to 0)\}\\
\cN &:= \{(r_\eps)_{\eps\in I}\in \K^I : \forall m\in \N\ |r_\eps| = O(\eps^{m}) \ (\eps\to 0)\}\\
\Kt &:= \cE_M / \cN
\end{split}
\]
\ed
Nets $(r_\eps)_\eps$ in $\cE_M$ are called moderate, those in $\cN$ negligible. 
We denote by $[(r_\eps)_\eps]$ the class of $(r_\eps)_\eps$ in $\Kt$.
$\Kt$ is a ring with all operations defined componentwise on representatives. 
It is, however, not a field: e.g., $[(\sin(1/\eps))]$ is a zero divisor. In 
fact,  by \cite[Th.\ 2.18]{AJ} or \cite[Th.\ 1.2.39]{book} we have:
\blem\label{zerodiv} For any $r\in\Kt$,
the following are equivalent:
\begin{itemize}
	\item[(i)] $r$ is not invertible.
	\item[(ii)] There exists a representative $(r_\eps)_\eps$ and a zero sequence $\eps_k$ with $r_{\eps_k}=0$ 
	for all $k$.
	\item[(iii)] $r$ is a zero divisor. 
\end{itemize}
\elem
Conversely, $r$ is invertible if and only if it is \emph{strictly nonzero}, i.e., iff there exists some
$m\in \N$ such that $|r_\eps|>\eps^m$ for $\eps$ small. Moreover, an element of $\Rt$ is called
\emph{strictly positive} if there exists some $m\in \N$ such that $r_\eps>\eps^m$ for $\eps$ small
(for further information on the order structure of $\Rt$ we refer to \cite{OPS:07}).

The foundations of the study of the algebraic properties of $\Kt$ were laid in \cite{AJ}. 
Important further milestones in this line of research are \cite{AJOS:08,AFJ09,Vernaeve:10}.
Next we list some fundamental properties of $\Kt$ and refer to the above works for proofs
and further results. Here and below we will use the notation $e_S$ for the equivalence class in $\Kt$ of 
the characteristic function of some $S\subseteq I$ (these classes were first introduced
and studied in \cite{AJ}). Then $e_S+e_{S^c}=1$ and $e_S\not=0$ if and only if $0\in \overline S$.

\begin{itemize}
\item   $\Kt$ is a reduced ring ($x^2=0 \Rightarrow x=0$).  
\item $e\in \Kt$ is idempotent ($e^2 = 1$) if and only if $e=e_S$ for some $S\subseteq I$.
\item $\Kt$ possesses uncountably many maximal ideals.
\item $\Kt$ is a complete topological ring.
\item The maximal ideals in $\Kt$ are precisely the closures of the prime ideals in $\Kt$.
\item Let $J$ be an ideal in $\Kt$. Then the closure of $J$ is the intersection
of all maximal ideals containing $J$.
\item $\Kt$ is {\em not:}
\begin{itemize}
\item Artinian
\item Noetherian
\item von Neumann regular
\end{itemize}
\item Every ideal $J$ in $\Kt$ is convex ($x\in J$, $|y|\le |x|$ $\Rightarrow$ $y\in J$).
\item An ideal $J$ is prime if and only if it is pseudoprime and radical, i.e.:
\begin{itemize}
\item $\forall S\subset (0,1]$: $e_S\in J$ or $e_{S^c}\in J$, and
\item $\forall x\in J$: $\sqrt{|x|}\in J$.
\end{itemize}
\item The minimal prime ideals are precisely the pure prime ideals.
\item The projective ideals are the ideals generated by a
family of mutually orthogonal idempotents.
\end{itemize}

We shall also require the following consequence of \cite[Lemma 2.3]{Vernaeve:10}:

\begin{lem}\label{VernaeveLemma} Let $\alpha, \beta \in \Kt$, then $\alpha\cdot
\beta = 0$ if and only if there exists a subset $S \subseteq I$ such that
$\al \cdot e_S = 0$ and $\be\cdot e_{S^c} = 0$. 
\end{lem}

Turning now to linear algebra, we first
recall a basic lemma (\cite[Lemma 1.2.41]{book}). By $M(n,\Kt)$ we denote the set of square 
matrices of size $n$ with entries in $\Kt$. 
 \begin{lem} \label{GKOS_Lemma} Let $A \in M(n,\Kt)$, then the following are
equivalent:
\begin{enumerate}
   \item $A$ is non-degenerate, i.e., $\forall x \in \Kt^n$: $x^t A y = 0$
$\forall y \in \Kt^n$ $\Rightarrow$ $x = 0$.

  \item $A$ is injective as a linear operator on $\Kt^n$.

  \item $A$ is bijective as a linear operator on $\Kt^n$.
  
  \item $\det(A)$ is invertible.
\end{enumerate}

 \end{lem}
 
Note that, for a square matrix over an arbitrary commutative ring with unit, (iii) and (iv) are equivalent, while
(ii) is equivalent to the determinant not being a zero divisor, which, by Lemma \ref{zerodiv}, amounts to the same.
There is, however, no well-defined notion of rank for matrices over $\Kt$:
Indeed, the $2\times 2$-matrix 
$$A=\begin{pmatrix} e_S & e_S\\ e_{S^c} & e_{S^c} \end{pmatrix}$$
has row rank $0$, but column rank $1$.
%\end{rem}

From \cite[Theorem 5.8]{EbM-LGiAGF} we obtain:
 \begin{lem} \label{complexEberhard} Let $v \in \Kt^n$ and $\langle\ , \
\rangle$ denote the standard euclidean or unitary inner product, then the
following are equivalent:
\begin{enumerate}
   \item $v$ is free.
   \item $\langle v, v \rangle$ is strictly positive.
   \item The coefficients of $v$ (with respect to any basis) span $\Kt$.
   \item $v$ can be extended to a basis of $\Kt^n$.
\end{enumerate}
\end{lem}
 
As is well-known, even in free modules of finite rank the extension of a given set
of free vectors to a basis might fail. (E.g., $2$ is free in the $\Z$-module $\Z$,
but cannot be member of any basis.) However, the above lemma implies that we may
always extend a single free vector in a free $\Kt$-module of finite rank to a
basis. In general, the problem of  basis extension in free modules over arbitrary
rings is related to the question of free quotient modules (cf.\ 
\cite[Chapter III, Section 3.3, Satz 10 (Erg\"anzungssatz)]{OR:74}). Fortunately,
the specific features of $\Rt$ and $\Ct$ allow us to prove the possibility of
general basis extension in any free $\Kt$-module of finite rank. 

\begin{lem}\label{extlemma}
Let $V$ be a free module over $\Kt$ of finite rank $n > 0$. Suppose that $k\in\N$,
$0 < k < n$ and that the set $\{v_1,\ldots,v_k\}$ is free in $V$. Then there
are $n-k$ vectors $v_{k+1},\ldots,v_n$ in $V$ such that $\{v_1,\ldots,v_n\}$ is a
basis of $V$. 
\end{lem}
\begin{proof} Since $V$ is isomorphic to $\Kt^n$, we may assume that $V = \Kt^n$.
Let $U := \Kt\text{-}\Span\{v_1,\ldots,v_k\}$ and let $(v_j^\eps)_{\eps \in I}$
 be an arbitrary representative of $v_j$ ($j=1,\ldots,k$). 
 
 We show that we can extend the set $\{v_1,\ldots,v_k\}$ by a vector 
$w \in U^\perp$ such that $\{v_1,\ldots,v_k,w\}$ is free in $\Kt^n$. 
 For every $\eps \in I$
we can find a vector $w_\eps \in \mathbb{K}^n$ satisfying $\langle w_\eps,
w_\eps \rangle = 1$ and $\langle w_\eps, v_j^\eps\rangle = 0$ ($j=1,\ldots,k$).
Since $\| w_\eps\| = 1$, the coefficients of $(w_\eps)_{\eps\in I}$ are
moderate nets in $\mathbb{K}$ and hence, by construction and Lemma
\ref{complexEberhard}, $w := [(w_\eps)]$ is a free vector in $U^\perp$. Hence
$\{v_1,\ldots,v_k,w\}$ is free. Proceeding by induction we obtain a basis after
$n-k$ steps (cf.\ Remark \ref{submodrem} below). 
\end{proof}

\begin{rem} Note that unlike in the vector space case $\mathbb{K}^n$ one
cannot expect to find a basis extension by an appropriate subset of vectors from
the standard basis $\{e_1,\ldots,e_n\}$. For example, consider the 
vector $v_1 = (c,1-c)$ in $\Rt^2$, where $c = [(c_\eps)]$ with $c_\eps = 1$, if
$1/\eps \in \N$, and
$c_\eps = 0$ otherwise. Then $v_1$ is a free vector, but neither $\{v_1,e_1\}$
nor $\{v_1,e_2\}$ constitutes a basis of $\Rt^2$, since the corresponding
determinants are zero divisors. However, $w = (1-c,c)$ is free
vector perpendicular to $v_1$ and $\{v_1,w\}$ is a basis of $\Rt^2$.
\end{rem}

Finally, we note the following basic characterization of non-degeneracy for 
bilinear forms on a free $\Kt$-module of finite rank.

 \blem \label{lem:nondeg form}
 Let $\omega$ be a bilinear form on a free $\Kt$-module $V$ of finite rank. 
 Then the following are equivalent:
 \begin{itemize}
 \item[(i)] $\omega$ is non-degenerate, i.e.,
 \beq\label{nondeg}
   \forall w \in V: \quad \omega(v,w) = 0 \;\; \forall v \in V 
    \Rightarrow w = 0.
\eeq
 \item[(ii)] The transpose $\omega^t$ of $\omega$, defined by
 $\omega^t(v,w)=\omega(w,v)$, is non-degenerate.
 \item[(iii)] The matrix of $\omega$ with respect to any basis of $V$ is invertible.
 \item[(iv)] The linear map $\omega^\flat:V\to V^*$,
 $\omega^\flat(v)\cdot w= \omega(v,w)$, is an isomorphism.
 \end{itemize}
 \elem
 \pr This follows from \cite{SS:V2}, Satz 70.3 and Satz 70.5, combined with Lemma \ref{zerodiv}.
 \ep

%%%%%%%%%%%%%%%%%%%%%%%%%%%%%%%%%%%%%%%%%%%%%%%%%%%%%%%%%%%%%%%%%%
 \section{Linear algebra and symplectic forms on $\Rt$-modules} \label{sec:lin alg}
%%%%%%%%%%%%%%%%%%%%%%%%%%%%%%%%%%%%%%%%%%%%%%%%%%%%%%%%%%%%%%%%%%

Our basic references on the general theory of modules over rings are
\cite{SS:V1,SS:V2} and \cite{OR:74}. As a guideline for elements of
symplectic linear algebra and geometry which are fundamental to applications in
microlocal analysis we use \cite[Sections 21.1 and 21.2]{Hoermander:V3}.

%%%%%%%%%%%%%%%%%%%%%%%%%%%%%%%%%%%%%%%%%%%%%%%%%%%%%%%%%%%%%%%%%%
 \subsection{Basic structure of symplectic $\Rt$-modules} \label{basicsec}
%%%%%%%%%%%%%%%%%%%%%%%%%%%%%%%%%%%%%%%%%%%%%%%%%%%%%%%%%%%%%%%%%%

\begin{defin} Let $V$ be an $\Rt$-module. An $\Rt$-bilinear form $\sig : V \times V \to \Rt$ is a \emph{symplectic form} on $V$ if $\sig$ is skew-symmetric, i.e., $\sig(v,w) = - \sig(w,v)$ for all $v, w \in V$, and non-degenerate. 
The pair $(V,\sig)$ is then called a \emph{symplectic $\Rt$-module.}
 \end{defin}
 
\begin{ex}\label{standardex} Our standard model space for a symplectic $\Rt$-module is $T^*({\Rt}^n) := \Rt^n \times \Rt^n$ with symplectic form $\tilde{\sig}$ given by 
\beq\label{standsymp}
  \tilde{\sig}((x,\xi),(y,\eta)) = \sum_{j=1}^n y_j \xi_j - 
  \sum_{j=1}^n x_j \eta_j
  = \dual{y}{\xi} - \dual{x}{\eta}
  \qquad \forall (x,\xi), (y,\eta) \in T^*(\Rt^n).
\eeq
Note that $T^*(\Rt^n)$ is a free module of rank $2n$ and possesses the basis $\{e_1,\ldots,e_n,f_1,\ldots,f_n\}$, where $e_j := (\de_j,0)$ and $f_j := (0,\de_j)$ with $\de_j$ the $j^\text{th}$ standard unit vector ($1 \leq j \leq n$), satisfying
$$
  \tilde{\sig}(e_j,e_l) = 0 = \tilde{\sig}(f_j,f_l), \quad \tilde{\sig}(f_j,e_l) = \delta_{jl}
  \quad (1 \leq j,l \leq n).
$$
\end{ex}

\begin{thm}\label{basisthm} Let $(V,\sig)$ be a symplectic $\Rt$-module, where $V$ is free and of finite rank $m \in \N$. Then $m$ is even, say $m = 2 n$, and $V$ possesses a \emph{symplectic basis}, i.e., a basis $\{e_1, \ldots,e_n,f_1,\ldots,f_n\}$ such that
\beq\label{sympbasis}
   \sig(e_j,e_l) = 0 = \sig(f_j,f_l), \quad \sig(f_j,e_l) = \delta_{jl}
  \quad (1 \leq j,l \leq n).
\eeq 
\end{thm}

\begin{proof} Let $\{b_1,\ldots,b_m\}$ be a basis of $V$. Note that we necessarily have $m > 1$, for $\sig$ is degenerate on any $\Rt$-span of a single vector by skew-symmetry.

\emph{Step 1:} We construct a $2$-dimensional submodule $S \subseteq V$ with a symplectic basis $\{e_1,f_1\}$.

Non-degeneracy of $\sig$ implies that the skew-symmetric Gramian matrix $G$ of $\sig$ with components $\sig(b_i,b_j)$ ($1 \leq i,j \leq m$) is invertible by Lemma \ref{lem:nondeg form}.
Set $f_1:=b_1$. In order to construct $e_1\in V$ with $\sigma(f_1,e_1)=1$ we denote by $[v]$ the coordinate representation
of $v\in V$ with respect to  $\{b_1,\ldots,b_m\}$. Then our task is to find $e_1$ with $[f_1]^t G [e_1] = 1$. To achieve this
it is sufficient to set $[e_1] := G^{-1}(1,0,\dots,0)^t$.
By applying $\sigma(f_1,\,.\,)$ or $\sigma(\,.\,,e_1)$ to $\lambda e_1 + \mu f_1 = 0$ it follows that $\lambda=\mu=0$. Thus
$\{e_1,f_1\}$ is free and so it forms a symplectic basis of its span $S$.

\emph{Step 2:} We show that $V = S \oplus S^\sig$, where $S^\sig := \{ w \in V \mid \sig(w,s) = 0 \; \forall s \in S \}$.

To begin with, we have $S \cap S^\sig = \{ 0 \}$, since $w \in S \cap S^\sig$ means $w = \la f_1 + \mu e_1$ with certain $\la, \mu \in \Rt$, while $\sig(w,f_1) = 0$ and $\sig(w,e_1) = 0$ imply $\la = \mu = 0$.

It remains to show that any $w \in V$ can be written in the form $w = w_1 + w_2$ with $w_1 \in S$ and $w_2 \in S^\sig$. We put $w_1 := \sig(w,e_1) f_1 - \sig(w,f_1) e_1 \in S$ and claim that $w_2 := w - w_1 \in S^\sig$ (in other words, $w \mapsto w_1$ is the projection onto $S$ along $S^\sig$). Let $v \in S$ have the basis  representation $v = \la f_1 + \mu e_1$, then a simple calculation yields
$$
  \sig(w_2,v) = \sig(w,\la f_1 + \mu e_1) - 
    \sig(\sig(w,e_1) f_1 - \sig(w,f_1) e_1,\la f_1 + \mu e_1) = 0.
$$

\emph{Step 3:} Let $\sig_1 := \sig \mid_{S^\sig \times S^\sig}$. We claim that $(S^\sig,\sig_1)$ is a free symplectic $\Rt$-module of rank $m-2$.

First, it is easily seen that $\sig_1$ is non-degenerate: If $v \in S^\sig$ and $\sig(v,w_2) = 0$ for all $w_2 \in S^\sig$, then also $\sig(v,w) = 0$ for all $w\in V$, since $V = S \oplus S^\sig$. Hence $v = 0$ by non-degeneracy of $\sig$. 

Second, we show that $S^\sig$ is a free submodule of rank $m-2$. For any $w \in V$ we have shown in Step 2 that $\pi(w) := w - \sig(w,e_1) f_1 + \sig(w,f_1) e_1$ belongs to $S^\sig$, thus  $\pi : V \to S^\sig$ defines an $\Rt$-linear map with $\pi(e_1) = \pi(f_1) = 0$. 
By Lemma \ref{extlemma} we have that the free set $\{e_1,f_1\}$ can be extended to a basis 
$\{e_1,f_1,b'_3,\ldots,b'_m\}$ of $V$.   
Now we put $c_l := \pi(b'_l)$ ($l=3,\ldots,m$) and show that $\{c_3,\ldots,c_m\}$ is a free set of vectors generating $S^\sig$: 

If $\sum_{l=3}^m \la_l c_l = 0$, then direct calculation gives
$$
  0 = \sum_{l=3}^m \la_l b'_l - \mu f_1 + \nu e_1,
$$  
where $\mu = \sum_{l=3}^m \la_l \sig(b'_l,e_1)$ and $\nu = \sum_{l=3}^m \la_l \sig(b'_l,f_1)$. Linear independence of $\{e_1,f_1,b'_3,\ldots,b'_m\}$ implies $\la_3 = \cdots = \la_m = 0$, thus the set $\{ c_3,\ldots,c_m\}$ is free.

To show that $\Rt\text{-span }\{ c_3,\ldots,c_m\} = S^\sig$ it suffices to note that $\pi$ is surjective. Indeed, $w\in S^\sig$ means $\sig(w,e_1) = \sig(w,f_1) = 0$ and therefore $\pi(w) = w$.

\emph{Step 4:} We complete the proof by induction on the rank $m$ of $V$.

We have seen that $m=1$ contradicts the non-degeneracy of $\sig$ and in case $m=2$ we have a symplectic basis by Step 1. The constructions in Steps 2 and 3 allow the case of rank $m$ to be reduced to that of rank $m-2$, where we have a symplectic basis by induction hypothesis, and the observation that the union of symplectic bases in $S$ and $S^\sig$ provides a symplectic basis for $V$. In particular, $m-2$ has to be an even number, hence so is $m$.
\end{proof}

\begin{rem}\label{nonfreeremark}\begin{trivlist}

\item{(i)}  In contrast to the above result, if  $(V,\sig)$ is a non-free,
finitely generated  $\Rt$-module, then we can never have a generating set of
vectors satisfying the relations \eqref{sympbasis}.  
In fact, \eqref{sympbasis} implies that $\{e_1,\dots,e_n,f_1,\dots,f_n\}$
is free (apply $\sigma(\,.\,,e_j)$ and $\sigma(\,.\,,f_j)$ to any linear combination).

\item{(ii)} A simple nontrivial example of a symplectic form on a non-free, finitely generated symplectic $\Rt$-module is provided by the following: Choose a zero divisor $\al \neq 0$  in $\Rt$ and define  
$$
  V := \Rt\text{-span }\{ (\al,0),(0,\al) \} \subseteq \Rt^2, 
  \quad \text{and } \sig := \tilde{\sig} \mid_{V \times V}. 
$$  
Clearly, $\sig$ is bilinear and skew-symmetric. Moreover, $\sig$ is also non-degenerate on $V$ by the following simple argument. Any vectors $v,w \in V$ are of the form $v = \al (x,\xi)$, $w = \al (y,\eta)$ with $(x,\xi), (y,\eta) \in \Rt^2$ and we obtain 
$$
  \sig(v,w) = \al (y \xi - x \eta).
$$
Thus, $\sig(v,w) = 0$ for all $w$ implies $\al x = 0$ and $\al \xi = 0$, hence $v = 0$. 

\item{(iii)} The fact that the underlying symplectic $\Rt$-modules are assumed to be free in most of the constructions to follow will not be a severe restriction for our first applications to non-smooth symplectic geometry. In fact, the typical module will then simply be the $\Rt$-extension of a classical tangent space $T_p M$ to a ($2n$-dimensional) manifold $M$ at the point $p \in M$ and will thus be isomorphic to $\Rt^{2n}$, hence free.

\end{trivlist}
\end{rem}

Besides constructing symplectic bases from scratch it is often important to extend
a given ``partial symplectic basis'', e.g., as discussed in \cite[Proposition
21.1.3]{Hoermander:V3} for the case of symplectic vector spaces. Thanks to Lemma
\ref{extlemma} we can also prove a symplectic basis extension result. 

\begin{prop}\label{basisextprop} Let $(V, \sig)$ be a free symplectic $\Rt$-module of finite rank $2n$. Let $I, J \subseteq \{ 1,\ldots n\}$ and $e_i \in V$ ($i\in I$) and $f_j \in V$ ($j \in J$) be elements such that the set $B:= \{e_i \mid i \in I \} \cup \{ f_j \mid j \in J\}$ is free in $V$ and satisfies
$$
  \sig(e_i,e_k) = 0 = \sig(f_j,f_l), \quad \sig(f_j,e_i) = \delta_{ji}
  \quad (i,k \in I; j,l \in J).
$$
Then we can find elements $e_i \in V$ ($i \in \{1, \ldots n\} \setminus I$) and
$f_j \in V$ ($j \in \{1, \ldots n\} \setminus J$) such that $\{e_1, \ldots, e_n,
f_1, \ldots, f_n\}$ is a symplectic basis of $(V,\sig)$.
\end{prop}
\begin{proof}  \emph{Case 1, $I = J$:} We assume $I \not= \{ 1,\ldots,n\}$, since
otherwise there is nothing left to be done. The equations in the hypothesis show
that $B$ is  a symplectic basis of a submodule $U$ (consisting of $2 |I|$ elements), i.e., the
restriction of $\sig$ to $U \times U$ defines a symplectic form on $U$. Exactly as
in step 3 of the proof of Theorem \ref{basisthm}, but employing Lemma
\ref{extlemma} to first obtain some basis of $V$ at all, one proves that
$U^\sig := \{ v \in V \mid \sig(v,w) = 0\; \forall w \in U\}$ is a symplectic
$\Rt$-module and free of rank $2 (n - |I|)$. Thus, by the same theorem, it
possesses a symplectic basis itself. Combining $B$ with the latter basis yields a
symplectic basis for $V$.

\emph{Case 2, $J\setminus I \not= \emptyset$:} Let $j_0 \in J\setminus I$. We will show that we can extend $B$ by a free element $e_{j_0} = e \in V$ satisfying 
$$
  \sig(e,e_i) = 0 \; (i \in I), \sig(e,f_j) = - \delta_{j_0 j} \; (j\in J).
$$
We may employ Lemma \ref{extlemma} and extend $B$ to a basis of $V$ by adding
appropriate elements $a_k \in V$ ($k\in \{1,\ldots,n\} \setminus I$) and $b_l \in
V$ ($l\in \{1,\ldots,n\} \setminus J$). The Gramian matrix of $\sig$ with respect
to this basis is invertible, hence the system of linear equations
\begin{align*}
  \sig(e,e_i) &= 0 \;(i\in I), &\sig(e,f_j) &= - \delta_{j_0,j} \; (j \in J),\\
  \sig(e,a_k) &= 0 \; (k \in \{1,\ldots,n\} \setminus I),&
  \sig(e,b_l) &= 0 \; (l \in \{1,\ldots,n\} \setminus J)
\end{align*}
is (uniquely) solvable for $e$. We claim that $B \cup \{e\}$ is free in $V$: Suppose
$$
  \la e + \sum_{i \in I} \la_i e_i + \sum_{j\in J} \mu_j f_j = 0
$$
holds with $\la, \la_1,\ldots,\la_n,\mu_1,\ldots,\mu_n \in \Rt$. Taking the vector on the left-hand side into the $\sig$-product with $f_{j_0}$ then yields $\la = 0$, which in turn implies $\la_i = 0$ ($i\in I$) and $\mu_j = 0$ ($j \in J$), since $B$ is free.

We may proceed in this way with extensions of $B$ until we reach the situation with $I = J$ and thus have reduced the proof to case 1.

\emph{Case 3, $I\setminus J \not= \emptyset$:} This is analogous to case 2 with the roles of the $e_i$'s and $f_j$'s exchanged. 
\end{proof}

We now define and study the structure-preserving maps between symplectic $\Rt$-modules.

\begin{defin} Let $(V,\sig)$ and $(W,\om)$ be symplectic $\Rt$-modules. An $\Rt$-linear map $f : V \to W$ is called \emph{symplectic} if for all $v_1, v_2 \in V$
$$
   \om(f(v_1),f(v_2)) = \sig(v_1,v_2).
$$
A symplectic isomorphism is called a \emph{symplectomorphism}.
\end{defin}

\begin{prop} Let $(V,\sig)$ and $(W,\om)$ be symplectic $\Rt$-modules and $f : V \to W$ be a symplectic map. Then the following hold:
\begin{trivlist}
\item{(i)} $f$ is injective.
\item{(ii)} If $V$ and $W$ are free and of equal finite rank, then $f$ is a symplectomorphism.
\end{trivlist}
\end{prop}

\begin{proof} (i): If $v \in \ker(f)$, then $\sig(v,u) = \om(f(v),f(u)) = 0$ for
all $u \in V$, hence $v = 0$ by non-degeneracy of $\sig$.

(ii): By (i) we already know that $f$ is injective. Let $V$ and $W$ be of rank $k$ and let $\cB$, $\cC$ denote a basis of $V$, $W$, respectively. These bases provide isomorphisms $\Phi_\cB : V \to \Rt^k$ and $\Phi_\cC : W \to \Rt^k$ and $A := \Phi_\cC \circ f \circ \Phi_\cB^{-1} : \Rt^k \to \Rt^k$  is injective. Thanks to Lemma \ref{GKOS_Lemma} $A$ is even bijective, hence $f$ is bijective.
\end{proof}

\begin{rem}\label{submodrem} Note that the argument in the final step of the above proof shows that the following property holds:  

($\star$) \emph{If $W$ is a free $\Rt$-module of finite rank and $U$ is a free submodule of the same rank as $W$, then $U = W$.} 

The proof employs Lemma \ref{GKOS_Lemma}, which relies on the special feature of
the ring $\Rt$ that invertibility is equivalent to not being a zero divisor. (In
general, injectivity of a linear map corresponding to a square matrix is just
equivalent to having a determinant which is not a zero divisor, cf.\
\cite[Korollar 48.8]{SS:V1} or \cite[Chapter III, Section 8,
Proposition 3]{bourbaki-algebra}.) In rings without such an equivalence the
corresponding property fails to hold: For example, the integral domain $\Z$
considered as $\Z$-module is free of rank $1$ and  $E := \{ 2 l \mid l \in \Z\}
\neq \Z$ is a free submodule of rank $1$ (with basis $\{2\}$). An example which is
not based on integral domains is provided by $W := \Cinf(\R)$ as a free module of
rank $1$ over $\Cinf(\R)$ and the free submodule $U$ of rank $1$ generated by the
free subset $\{g\}$, where $g(x) = x$. We have $U \neq W$, since $f \in U$ implies
$f(0) = 0$. 
\end{rem}

Every finite dimensional symplectic vector space over $\R$ is symplectomorphic to a standard phase space $T^*(\R^n)$. 
In general, a finitely generated symplectic $\Rt$-module $(V,\sig)$ will not be symplectomorphic to a model example \ref{standardex}, in fact, in case $V$ is not free it cannot be (cf.\ Remark \ref{nonfreeremark}(i)). However, if $V$ is free and of finite rank we have an analogue of the classical result.

\begin{cor} A free symplectic $\Rt$-module $(V,\sig)$ of rank $2n$ is symplectomorphic to $(T^*(\Rt^n),\tilde{\sig})$.
\end{cor}

\begin{proof} By Theorem \ref{basisthm} we may choose a symplectic basis $\{e_1,\ldots,e_n,f_1\ldots,f_n\}$ of $(V,\sig)$. Defining an $\Rt$-linear map $V \to \Rt^n \times \Rt^n$ by $e_j \mapsto (\de_j,0)$ and $f_j \mapsto (0,\de_j)$ ($j=1,\ldots,n$) and $\Rt$-linear extension we obtain an isomorphism which is a symplectic map by construction.  
\end{proof}

\paragraph{Matrix of a symplectomorphism:} Let $f : V \to W$ be a symplectomorphism between free symplectic $\Rt$-modules of rank $2n$. Choosing symplectic bases in $V$ and $W$ we obtain a matrix $A \in M(2n,\Rt)$ representing $f$, where $A$ is a symplectomorphism of $(T^*(\Rt^n),\tilde{\sig})$. We may now proceed as in the case of $\R$-vector spaces by writing
$$
   \tilde{\sig}((x,\xi),(y,\eta)) = 
    \begin{pmatrix} x & \!\!\! \xi \end{pmatrix}  J 
    \begin{pmatrix} y\\ \eta \end{pmatrix},
    \quad \text{ where } J = 
    \begin{pmatrix} 0 & -I_n\\ I_n & 0 \end{pmatrix},
$$
and using the symplecticity condition $A^* \tilde{\sig} = \tilde{\sig}$ to deduce the following matrix relation in $M(2n,\Rt)$
\beq\label{sympmatrix}
   A^t J A = J,
\eeq
which implies $\det(A)^2 = 1$. Moreover, if $\la \in \Ct$ we obtain from \eqref{sympmatrix} (with $I$ now denoting the identity matrix in $M(2n,\Ct)$)
\beq\label{sympdeteq}
  \det(A - \la I) = \det(J^{-1} (A^t)^{-1} J - \la I) =
   \det(J^{-1} ( (A^{-1})^t - (\la I)^t) J) 
   %= \det( (A^{-1})^t - (\la I)^t) 
   = \det( A^{-1} - \la I).
\eeq

We consider various notions mimicking those of significant types of subspaces of symplectic vector spaces.

\begin{defin} Let $(V,\sig)$ be a symplectic $\Rt$-module and $U$ be a submodule of $V$.
\begin{enumerate}

\item If $A \subseteq V$ is an arbitrary subset, then $A^\sig := \{ v \in V \mid \sig(v,u) = 0 \; \forall u \in A\}$ is a submodule and is called the \emph{orthogonal} or \emph{annihilator} (with respect to $\sig$) of $A$.

\item $U$ is a \emph{symplectic} submodule if $\sig \!\!\mid_{U \times U}$ is non-degenerate (equivalently, $U \cap U^\sig = \{0\}$), i.e., the restriction of $\sig$ defines a symplectic form on $U$.

\item $U$ is an \emph{isotropic} submodule if $U \subseteq U^\sig$.

\item $U$ is an \emph{involutive} (or \emph{coisotropic}) submodule if $U \supseteq U^\sig$.

\item $U$ is a \emph{Lagrangian} submodule if $U = U^\sig$.

\end{enumerate}
\end{defin}

Symplectic submodules provide direct sum decompositions, a fact that has implicitly been used in the proof of Theorem \ref{basisthm}.

\begin{prop} Let $(V,\sig)$ be a symplectic $\Rt$-module and let $U$ be a submodule of $V$. Then the following hold:
\begin{enumerate}

\item If $V = U \oplus U^\sig$, then $U$ is a symplectic submodule.

\item If $U$ is free of finite rank and symplectic, then $V = U \oplus U^\sig$.

\end{enumerate}
\end{prop}

\begin{proof} (i) is clear, since the direct sum decomposition requires $U \cap U^\sig = \{ 0 \}$. 

(ii):  Since $U$ is symplectic we have $U \cap U^\sig = \{ 0\}$. We employ Theorem \ref{basisthm} and let $U$ have rank $2k$ and $\{ e_1,\ldots,e_k,f_1,\ldots,f_k\}$ be a symplectic basis  of $U$. Then we define the projection $\pi : V \to U$ by
$$
    \pi(v) := \sum_{j=1}^k \sig(v,e_j) f_j - 
    \sum_{j=1}^k \sig(v,f_j) e_j \qquad (v \in V).
$$
By direct calculation one verifies that $v - \pi(v) \in U^\sig$, hence $v =  \pi(v) + (v - \pi(v))$ for any $v \in V$ which proves that $V = U + U^\sig$. 
\end{proof}

Although we cannot have a direct sum decomposition as above with non-symplectic submodules, we still have a general equation for the ranks.

\begin{prop} Let $(V,\sig)$ be a free symplectic $\Rt$-module of finite rank. If
$U$ is a free submodule, then $U^\sig$ is free and
$$
  \text{rank} (U) + \text{rank} (U^\sig) = \text{rank} (V). 
$$
\end{prop}

\begin{proof} Let $V$ be of rank $2n$ and $U$ be of rank $k \leq 2n$. The cases
$k=0$ or $k=2n$ are trivial, thus we assume $0 < k < 2n$. Let $\{b_1,\ldots,b_k\}$
be a basis of $U$, which we extend by $\{b_{k+1},\ldots,b_{2n}\}$ to a basis of
$V$ (using Lemma \ref{extlemma}). Define the $\Rt$-linear map $f\col V \to V$ by
$$
  f(v) := \sum_{j=1}^{2n} \sig(v,b_j) b_j. 
$$
The matrix of $f$ with respect to the basis $\{b_1,\ldots,b_{2n}\}$ is exactly the Gramian matrix $ G = (\sig(b_i,b_j))_{1\leq i,j\leq 2n}$ of $\sig$. Thus, nondegeneracy of $\sig$ and Lemma \ref{GKOS_Lemma} imply that $f$ is an isomorphism.

We claim that $f(U^\sig) = W$, where $W := \Rt\text{-span
}\{b_{k+1},\ldots,b_{2n}\}$: Since $\sig(v,b_i) = 0$ for all $i=1,\ldots,k$ and
$v\in U^\sig$, we clearly have $f(U^\sig) \subseteq W$. To prove the reverse
inclusion, let $w \in W$ arbitrary.  
Since $f$ is bijective, there exists
a unique $v= \sum_{l=1}^{2n} x_l b_l\in V$ with $f(v)=w$ and it remains to show that $v\in U^\sigma$.
Write the basis expansion of $w$ as $w =
\sum_{j=k+1}^{2n} \la_j b_j$, 
and let $u = \sum_{i=1}^k \mu_i b_i \in U$ be arbitrary.
Then in matrix notation with $x = (x_1,\ldots,x_{2n})^t$ we obtain $G^t x =
(0,\ldots,0,\la_{k+1},\ldots,\la_{2n})^t$. Thus 
$$
  \sig(v,u) = x^t \cdot G\cdot (\mu_1,\ldots,\mu_k,0,\ldots,0)^t = (0,\ldots,0,\la_{k+1},\ldots,\la_{2n}) \cdot (\mu_1,\ldots,\mu_k,0,\ldots,0)^t = 0,
$$
hence $v \in U^\sig$ and $f(U^\sig) \supseteq W$.  

We have shown that $U^\sig$ is isomorphic to the free submodule $W$ under $f$, hence $U^\sig$ itself is free and of rank $2n - k$.
\end{proof}

As in the case of symplectic vector spaces, the general rank equation in the above proposition allows for a convenient characterization of Lagrangian submodules.

\begin{cor} Let $(V,\sig)$ be a free symplectic $\Rt$-module of finite rank. If
$U$ is a free submodule, then the following are equivalent:
\begin{enumerate}
\item $U$ is Lagrangian.

\item $U$ is isotropic and $\text{rank}(U) = \text{rank}(V)/2$.
\end{enumerate}
\end{cor}
\begin{proof} (i) $\Rightarrow$ (ii): We have $U = U^\sig$ by definition and the above proposition then implies $ \text{rank} (V) = \text{rank} (U) + \text{rank} (U^\sig) = 2 \,\text{rank}(U)$.

(ii) $\Rightarrow$ (i): The hypothesis on the rank and the above proposition yield $\text{rank}(U^\sig) = \text{rank}(U)$. Isotropy means $U \subseteq U^\sig$, thus we obtain from the observation ($\star$) in Remark \ref{submodrem} that $U = U^\sig$.
\end{proof}

The simplest example of a Lagrangian submodule is of course $U = \Rt^n \times
\{0\}$ as a submodule of $(T^*(\Rt^n),\tilde{\sig})$ and the following theorem
shows that on an abstract level this is the standard form of a Lagrangian. Recall
that the dual $M^*$ of a module $M$ over the ring $R$ consists of the $R$-linear
maps from $M$ to $R$ and that the concept and construction of dual bases in $M^*$
exist for free modules of finite rank in analogy to the case of finite dimensional
vector spaces.

\begin{thm} Let $(V,\sig)$ be a free symplectic $\Rt$-module of finite rank and
$U$ be a free submodule. If $U$ is Lagrangian, then $(V, \sig)$ is
symplectomorphic to $(U \oplus U^*,\om)$, where 
$$ 
 \om(u\oplus\alpha,v\oplus\beta) := \alpha(v) - \beta(u)  .
$$
\end{thm}

\begin{proof} Let $V$ be of rank $2n$, then $U$ has rank $n$. Let $\{ e_1,\ldots, e_n\}$ be a basis of $U$. Since $U^\sig = U$, we are in the situation of Proposition \ref{basisextprop} with $I = \{1,\ldots,n\}$ and $J = \emptyset$ and may extend the basis of $U$ to obtain a symplectic basis $\{e_1,\ldots,e_n,f_1,\ldots,f_n\}$ of $(V,\sig)$.

Let $\{e_1^*,\ldots,e_n^*\}$ be the basis of $U^*$ which is dual to $\{e_1,\ldots,e_n\}$ and define $\Phi \col V \to U \oplus U^*$ by $\Rt$-linear extension of the assignments $e_i \mapsto (e_i,0)$ and $f_j \mapsto (0,e_j^*)$. Since $\Phi$ maps a basis to a basis, it is an isomorphism of $\Rt$-modules. Moreover, $\{(e_1,0),\ldots,(0,e_n^*)\}$ is by construction an $\om$-symplectic basis of $U\oplus U^*$ (which also proves that $\om$ is a symplectic form), hence $\Phi$ is a symplectomorphism.
\end{proof}

%%%%%%%%%%%%%%%%%%%%%%%%%%%%%%%%%%%%%%%%%%%%%%%%%%%%%%%%%%%%%%%%%%
 \subsection{Matrices over $\Ct$: symmetry, eigenvalues and spectral properties} \label{matrixsec}
%%%%%%%%%%%%%%%%%%%%%%%%%%%%%%%%%%%%%%%%%%%%%%%%%%%%%%%%%%%%%%%%%%

When trying to introduce a general concept of eigenvalues of matrices $A \in M(n,\Ct)$ a first natural
idea is to consider generalized complex numbers $\la$ for which $\det(A - \la I)$ is a zero divisor in $\Ct$, 
or equivalently, for which $\det(A - \la I)$ is not invertible.

By Lemma \ref{GKOS_Lemma}, this is equivalent to 
$\ker(A - \la I) \not= \{0\}$. In fact, the same
condition is used, e.g., in \cite[Chapter V, Section 7, Definition 1]{OR:74} to
define the notion of eigenvalue for matrices over arbitrary commutative
rings with unit. Moreover, every vector in $\ker(A - \la I) \setminus \{0\}$ is
then also called eigenvector. However, cautioned by somewhat pathological
effects due to zero divisors (see Example \ref{diagex}) we will use
a different definition, see Def.\ \ref{evdef} below.

\begin{ex}\label{diagex} Let $c \in \Rt$ such that $c \not= 0$, $c \not=1$, and $c
(1-c) = 0$.\footnote{For example, $c = e_S$ with $S \subseteq I$ such that $0 \in
\overline{S}$ and $0 \in \overline{S^c}$ would do.} Consider the matrix $A =
\begin{pmatrix} 1-c & 0\\ 0 & c   \end{pmatrix}$, then we have
$$ 
(\star) \qquad  \det(A - \la I) = (1 - c - \la)(c - \la)  = \la^2 - \la + c (1-c)
= \la(\la-1).
$$
 
Hence we obtain the following equivalence: 
$\det(A - \la I)$ is not invertible
iff  
$\la$ or $\la - 1$ is a zero divisor. 
Note that only in case $\la (\la - 1) = 0$ (including the instances
$\la = 0$, $\la = 1$, $\la = c$, and $\la = 1-c$) do we obtain $\det(A - \la I) =
0$. Furthermore, we
observe that, regardless of the precise choice of $c$, any zero divisor
$\la \in \Ct$ produces a non-invertible $\det(A - \la I)$.   %is a spectral value of $A$.
We investigate the submodules $E(\la) := \ker(A - \la I) \subseteq \Ct^2$ for the
various cases of $\la$ with this property: We have $E(\la) = \{ (x,y) \mid (1-c-\la)x
= 0, (c-\la)y = 0 \}$. 

\begin{trivlist}

\item{\underline{Case 1, $\la (\la - 1) = 0$:}} The vector $v_\la :=
(c-\la,1-c-\la)$ belongs to
$E(\la)$ as is seen directly from ($\star$). Moreover, $v_\la$ is free,
since $\mu (c-\la) = 0$, $\mu (1-c-\la) = 0$ implies $\mu = \mu c + \mu \la = 2
\mu c$, hence $\mu c = 2 \mu c^2 = 2 \mu c$, which means $\mu c = 0$ and in turn
yields $\mu = 0$. We will show that $E(\la) = \Span \{ v_\la \}$. 

First, we extend
$v_\la$ to a basis of $\Ct^2$ by the vector $w := (1-c-\la,c-\la)$. To see that
$\{ v_\la, w \}$ is a basis we note that $\det(v_\la \; w) = (2c - 1)(1 - 2
\la)$ is invertible, since $1 - 2 r$ is an invertible number for any $r$
satisfying $r(r-1) = 0$.\footnote{If $\mu (1-2r)=0$, then $\mu = 2 r \mu$, hence
$\mu r = 2 \mu r^2 = 2 \mu r$, i.e., $\mu r = 0$ and thus $\mu = 0$.} Second, we
use the basis representation $(x,y) = r v_\la + s w$ and ($\star$) in the form $(1
- c - \la)(c - \la) = \la (\la - 1) = 0$ to obtain that $(x,y) \in E(\la)$ implies
$0 = s (1-c-\la)^2 = s ((1 - 2c)(1 - 2\la) + (c-\la)^2)$ and $0 =  s (c-\la)^2$.
Adding the
equations gives $0 = s (1-2c)(1 - 2\la)$ and thus $s = 0$ by invertibility of
$(1-2c)(1 - 2\la)$.
Therefore $(x,y) \in \Span \{ v_\la \}$.

\item{\underline{Case 2, $\la (\la - 1) \neq 0$:}} We claim that there is no free
vector in $E(\la)$. Indeed any $v \in E(\la)$ is of the form $v = (r,s)$ with $r
(1-c-\la) = 0$ and $s (c-\la) = 0$, hence $t v = (0,0)$, where $t :=
(c-\la)(1-c-\la) = \la (\la - 1) \neq 0$ by ($\star$).

\end{trivlist}

\end{ex}

A further reason for not adopting the above concept as a definition of eigenvalues is the
observation that non-invertibility of $\det(A - \la I)$ implies that there exists some
$e_S\not=0$ with $\det(A - \la I)\cdot e_S=0$. But then for {\em any} $\mu\in \Ct$, also
$\det(A - (\la + \mu e_{S^c}) I)$ is non-invertible, despite the fact that $\mu$ has no 
relation to the matrix $A$.

\begin{rem} In  \cite[Definition 4.5]{EbM-LGiAGF} a definition of ordered sets of
\emph{generalized eigenvectors} of a symmetric matrix $A$ over $\Rt$ was given
in terms of the sizes of the real parts of eigenvalues of a representing net
$(A_\eps)$ (cf.\ Sec.\ \ref{skewsec} below). In this sense, the matrix $A$ in the above example has the generalized
eigenvalues $0$ and $1$. However, this definition does not include all cases of
numbers $\la$ such that we can find a free vector $v$ satisfying $A v = \la
v$, e.g., $\la = c$ with $v = (0,1)$ and $\la = 1-c$ with $v = (1,0)$ are
missing.
\end{rem}

In view of the above considerations we introduce the notion of eigenvalue using a stronger
condition than the non-invertibility of $\det(A - \la I)$.%above definition of spectral value. 

\begin{defin}\label{evdef} Let $A \in M(n,\Ct)$. A generalized complex number $\la$ is an
\emph{eigenvalue} of $A$ if there exists a free vector $x \in \Ct^n$ such that $A
x = \la x$. The free vector $x$ is then called an \emph{eigenvector} for $A$
associated with $\la$.
\end{defin}

We note that by Lemma \ref{complexEberhard} (ii) this definition agrees with the one introduced by H.\ Vernaeve in \cite[Def.\ 4.8]{Vernaeve:08}.

\begin{prop} \begin{enumerate}

\item If $A \in M(n,\Ct)$ is Hermitian and $\la$ is an eigenvalue of $A$, then $\bar{\la} = \la$, i.e., $\la \in \Rt$.

\item If $A \in M(n,\Ct)$ is skew-Hermitian and $\la$ is an eigenvalue of $A$, then $\la^2 \in \Rt$.

\end{enumerate}
\end{prop}
\begin{proof} 
To prove (i) let $x \in \Ct^n$ be a free vector such that $A x = \la x$ and let
$\dual{y}{z} = \sum_{j=1}^n y_j \bar{z_j}$ denote the standard Hermitian form on
$\Ct^n$. Since $x$ is free,  Lemma \ref{complexEberhard} yields that $\dual{x}{x}$
is strictly positive, hence invertible. Hence we may follow the classical
argument
$$
  \la \dual{x}{x} = \dual{Ax}{x} = \dual{x}{Ax} = 
  \bar{\la}\dual{x}{x}
$$ 
and deduce that $\bar{\la} = \la$. Clearly, (ii) follows from (i), since $\la^2$
is then an eigenvalue of the Hermitian matrix $A^2$.
\end{proof}
 For any eigenvalue of $A$, 
$\det(A - \la I)$ is not invertible (again by Lemma \ref{GKOS_Lemma}). 
The next proposition will show that in fact $\la$ is an eigenvalue if and only if  $\det(A - \la I)=0$.  
In its proof we will use the following observation.
\begin{rem}\label{dieda} Let $A: \Rt^n\to \Rt^n$ linear and $B:\Ct^n\to \Ct^n$ its complexification,
$Bz:=Ax+iAy$ for $z=x+iy$. Then 
\begin{itemize}
\item[(i)] $\ker B = \ker A + i \ker A$.
\item[(ii)]	$\exists v$ free with $v\in \ker A$ $\Leftrightarrow$ $\exists z$ free with $z\in \ker B$.

To see this, note first that given $v$ free in $\ker A$, $z:=v+i0$ is free in $\ker B$. Conversely,
let $z$ be a free vector in $\ker B$, and assume without loss that $\|z\|=1$. Write $z=x+iy$, so that
$x$, $y\in \ker A$. Then since $1=\|z\|^2=\|x\|^2+\|y\|^2$, for any representative $z_\eps=x_\eps+iy_\eps$
there exists some $\eps_0>0$ such that $\forall \eps<\eps_0$, $\|x_\eps\|^2+\|y_\eps\|^2>\frac{1}{2}$.
Now for any $\eps\in (0,\eps_0]$, we define $\alpha_\eps$, $\beta_\eps$ as follows:
\begin{itemize}
\item if $\|x_\eps\|^2>\frac{1}{4}$, set $\alpha_\eps=1$, $\beta_\eps=0$
\item if $\|x_\eps\|^2\le\frac{1}{4}$, set $\alpha_\eps=0$, $\beta_\eps=1$
\end{itemize}
Then $(\alpha_\eps)$, $(\beta_\eps)$ are representatives of generalized numbers $\alpha$, $\beta$
and $v:=\alpha x + \beta y$ satisfies $\|v\|^2\ge\frac{1}{4}$, hence is free, and $Av=\alpha Ax+\beta Ay=0$.
\end{itemize}
\end{rem}

\begin{prop}\label{evchar} 
\begin{enumerate}

\item  $\la$ is an eigenvalue of $A \in M(n,\Ct)$ if and only if $\det(A - \la I) = 0$.

\item Eigenvalues of invertible matrices are invertible.
\item  
If $A$ is a symplectic matrix then $\la$ is an eigenvalue of $A$ if and only if
$\la^{-1}$ is. 
\end{enumerate}
\end{prop}\label{evprop}
\begin{proof} (i): Given an eigenvalue $\la$, choose a free vector $x \in \Ct^n$ such that $A x = \la x$. By Lemma \ref{complexEberhard} we may extend $\{ x\}$ to a basis $\cB$ of $\Ct^n$, where $x$ occurs as the first basis vector. The matrix representation of $A - \la I$ with respect to the basis $\cB$ has the zero vector as its first column, hence $\det(A - \la I) = 0$. (Alternatively, we may note that by \cite[Korollar 48.8]{SS:V2} for a general ring $R$ we have: if $x \in R^n$,  $B \in M(n,R)$, and $B\cdot x = 0$, then $\det(B) \cdot x = 0$. In our case (with $B:=A-\la I$), $x$ is free, implying $\det(B) = 0$.)

Conversely, let $\det(A - \la I) = 0$ and set $B:=A-\la I$. We claim that there exists some $v\in \Ct^n$ with
$\|v\|=1$ (hence free) such that $Bv=0$. By Rem.\ \ref{dieda} (ii) this will give the result.
Let $(B_\eps)_\eps$ be some representative of $B$.
Then for each $\eps\in I$ the polynomial $\det (B_\eps-\mu I)$ factors over $\C$ with zeros $\mu_{i\eps}$ $(i\in \{1,\dots,n\})$
and we may assume that $|\mu_{1\eps}|\le |\mu_{2\eps}|\le \dots \le |\mu_{n\eps}|$. For each $\eps$,
let $v_\eps$ be an eigenvector of $B_\eps$ to the eigenvalue $\mu_{1\eps}$ with $\|v_{\eps}\|=1$. Then $v:=[(v_\eps)_\eps]$ has
norm $1$. Moreover, $(\mu_{1\eps})\in \cN$: otherwise, there would exist some sequence $\eps_j\to 0$ and some
$q$ such that $|\mu_{1\eps_j}|\ge \eps_j^q$ for all $j$. But then 
$$
|\det B_{\eps_j}| = |\mu_{1\eps_j}|\dots |\mu_{n\eps_j}| \ge \eps_j^{q n}
$$
for all $j$, contradicting $\det B = 0$. It follows that $\|B_\eps v_\eps\| = \|\mu_{1\eps} v_\eps\|$ is negligible
and therefore $Bv=0$ in $\Ct^n$, as claimed.

  For an alternative proof of (i) we refer to \cite[Prop.\ 4.9]{Vernaeve:08} 

(ii): Let $A$ be an invertible matrix and let $\la$ be an eigenvalue of $A$. By (i) we have $\det(A - \la I) = 0$.
Suppose that $\la$ is not invertible. Then $\la$ is a zero divisor and we can find $\mu \in \Ct$, $\mu \neq 0$, such that $\mu \la = 0$. We obtain
$$
   \mu^n \det(A) = \det(\mu A) = \det(\mu A - \mu \la I) 
   = \mu^n \det(A - \la I) = 0,
$$
hence $\mu^n = 0$, since $\det(A)$ is invertible. Since there are no (non-zero) nilpotent elements in $\Ct$, we conclude that $\mu = 0$, which is a contradiction.

(iii) This is immediate from (i) and \eqref{sympdeteq}. 
\end{proof}
\begin{ex}
(i)
As is immediate from Prop.\ \ref{evchar}, if $\la_1,\ldots,\la_n \in \Ct$, $B = \text{diag}(\la_1,\ldots,\la_n)$, 
and $\la \in \Ct$, then:
$$
  \la \text{ is an eigenvalue of } B \quad\Longleftrightarrow\quad 
   \mathop{\Pi}\limits_{k=1}^n (\la - \la_k) = 0. \quad
$$ 
(ii) Let $A\in M(n,\K)$ be a classical matrix. 
Then the generalized eigenvalues of $A$ are precisely the interleavings of 
the classical eigenvalues of $A$. Here, by an interleaving of an $m$-tuple
$(\lambda_1,\dots,\lambda_m)$ we mean any $\lambda\in \Kt$ such that
there exists a map $j: I \to \{1,\dots,m\}$ with $\lambda = [(\lambda_{j(\eps)})_\eps]$
(cf.\ \cite{OV:08}). This follows from Prop.\ \ref{evchar} (i), together with the observation
that any polynomial with coefficients in $\K$ has as zeros in $\Kt$ precisely
the interleavings of its classical zeros (cf.\ the proof of \cite[Prop.\ 2.12]{O:92}). 
\end{ex}

A far-reaching generalization of (ii) in the previous remark is the following result
due to H.\ Vernaeve:

\bthm \label{vmain} Let $A\in M(n,\Ct)$ and denote by $p(\lambda) = \lambda^n + a_{n-1}\la^{n-1} + \dots +a_1 \la + a_0$
$(a_0,\dots,a_{n-1} \in \Ct)$ the (normalized) characteristic polynomial $(p(\lambda)=\pm \det(A-\la I))$.
Then there exist $\la_1,\dots,\la_n \in \Ct$ such that $p(\la) = (\la - \la_1)\dots (\la-\la_m)$.
Moreover, given any $n$-tuple $(\lambda_1,\dots,\lambda_n)$ with this property, any eigenvalue
$\la$ of $A$ is an interleaving of $(\lambda_1,\dots,\lambda_n)$:
$$
\exists \{S_1,\dots,S_n\} \ \text{partition of} \ (0,1]: \la = \la_1 e_{S_1} + \dots \la_n e_{S_N}
$$
\ethm
\begin{proof} This is immediate from \cite[Lemma 4.7]{Vernaeve:08} and Prop.\ \ref{evchar} (i).
\end{proof}
This result completely clarifies the structure of the set of eigenvalues of a general matrix
$A\in M(n,\Ct)$. In the subsequent sections we will see that for certain specific classes of matrices
(e.g., symmetric or skew-symmetric) it is possible to uniquely single out distinguished $n$-tuples 
(e.g., with certain order properties) as in Th.\ \ref{vmain}. The only additional degree of freedom
in the set of all eigenvalues is then induced by the interleavings of these distinguished $n$-tuples.

%%%%%%%%%%%%%%%%%%%%%%%%%%%%%%%%%%%%%%%%%%%%%%%%%%%%%%%%%%%%%%%%%%
 \subsection{Skew-symmetric and Hermitian matrices} \label{skewsec}
%%%%%%%%%%%%%%%%%%%%%%%%%%%%%%%%%%%%%%%%%%%%%%%%%%%%%%%%%%%%%%%%%%
Motivated by applications in low-regularity Riemannian and Lorentzian geometry,
a thorough study of symmetric matrices over $\Rt$ was carried out in
\cite{EbM-LGiAGF}. As the Gramian matrix of a symplectic form is skew-symmetric, 
in this section we present a similar analysis of skew-symmetric generalized
matrices.

We call a matrix $A\in M(n,\Rt)$ skew-symmetric if $A^t=-A$. The first basic
observation is that any skew symmetric matrix possesses a skew-symmetric
representative:

 \blem \label{lem:s-s representative}
 Let $A\in M(n,\Rt)$. Then the following are equivalent:
 \begin{itemize}
 \item[(i)] $A$ is skew-symmetric.
 \item[(ii)] There exists a skew-symmetric representative
 $(A_\eps)_\eps=((a_{ij}^\eps)_{i,j})_\eps$ of $A$.
 \end{itemize}
 \elem

 \pr It clearly suffices to prove (i)\ $\Rightarrow$ (ii). To this end, take an arbitrary representative
 $((\tilde{a}_{ij}^\eps)_{i,j})_\eps$ of $A$, and set
 $a_{ij}^\eps := \frac{1}{2}(\tilde{a}_{ij}^\eps-\tilde{a}_{ji}^\eps)$.
 Then $(a_{ij}^\eps-\tilde{a}_{ij}^\eps)_\eps\in \cN$ by skew-symmetry of $A$,
 so $(a_{ij}^\eps)_\eps$ is the desired representative.
  \ep
  
In \cite{EbM-LGiAGF}, a specific notion of eigenvalues for symmetric matrices
in $M(n,\Rt)$ is defined.  
Given $A=A^t\in M(n,\Rt)$, with
representative $(A_\eps)_\eps$, for each
$\eps$ let $\theta_{k,\eps}=\mu_{k,\eps} + i \nu_{k,\eps}$ ($1\le k\le n$)
be the eigenvalues of $A_\eps$, ordered by the size of their real parts:
$\mu_{1,\eps}\ge \dots \ge \mu_{n,\eps}$. 
Then Mayerhofer calls the generalized numbers $\theta_k := [(\theta_{k,\eps})_\eps]$ ($1\le i \le
n$) the eigenvalues of $A$. By Prop.\ \ref{evchar} (i) the $\theta_k$ are eigenvalues of $A$ in the sense of Def.\ \ref{evdef}. 
Moreover, by Th.\ \ref{vmain}, any other 
eigenvalue of $A$ is an interleaving of the $\theta_k$.

For any symmetric $A\in M(n,\Rt)$, the $\theta_k$ as above are well-defined real 
generalized numbers, independent of the representative of $A$ used for defining them 
(\cite[Lemma 4.6]{EbM-LGiAGF}).   The proof of this fact 
relies on the following numerical estimate for arbitrary perturbations
of Hermitian matrices.

\blem \label{evest}
Let $A\in M(n,\C)$ be a Hermitian matrix with eigenvalues $\lambda_1\ge \dots \ge \lambda_n$
and let $B$ be an arbitrary matrix in $M(n,\C)$ with eigenvalues $\beta_1,\dots,\beta_n$ such
that $\mathrm{Re}(\beta_1) \ge \dots \ge \mathrm{Re}(\beta_n)$. Then for a constant $C_n$ (depending
only on the dimension $n$)
$$
\max_{1\le k \le n} |\lambda_k - \beta_k| \le C_n\|A-B\|
$$
\elem
\begin{proof} See \cite[Th.\ 23.3]{B:87}.
\end{proof}

We may utilize Lemma \ref{evest} to obtain a distinguished $n$-tuple of eigenvalues
for any Hermitian matrix over $\Ct$:
\bprop \label{Hermevwelldef}
Let $A=[(A_\eps)_\eps]\in M(n,\Ct)$ be Hermitian.
Let $\lambda_k^\eps=\mu_{k,\eps} + i \nu_{k,\eps}$ ($1\le k\le n$)
be the eigenvalues of $A_\eps$, ordered by the size of their real parts:
$\mu_{1,\eps}\ge \dots \ge \mu_{n,\eps}$. 
Then $\lambda_k:=[(\lambda_k^\eps)_\eps]$ are well-defined elements 
$\lambda_1\ge \dots \ge\lambda_n$ of $\Rt$, independent of the representative
$(A_\eps)_\eps$. 
Any eigenvalue of $A$ is an interleaving of $\{\la_1,\dots,\la_n\}$.
Moreover, there exists a unitary matrix
$U\in M(n,\Ct)$ such that $U^*AU = \text{diag}(\lambda_1,\dots,\lambda_n)$.
\eprop
\begin{proof} Analogously to Lemma \ref{lem:s-s representative} it follows that
$A$ possesses a representative $(A_\eps)_\eps$ consisting entirely of Hermitian matrices.
Denote by $\lambda_1^\eps\ge \dots \ge\lambda_n^\eps$ the eigenvalues
of $A_\eps$. 
For any other representative $(B_\eps)_\eps$ of $A$ with eigenvalues $\beta_1^\eps,\dots,\beta_n^\eps$
ordered by the size of their real parts, Lemma \ref{evest} implies that $(\lambda_k^\eps)_\eps$
and $(\beta_k^\eps)_\eps$ define the same equivalence class in $\Ct$ ($1\le k \le n$).
This shows that $\lambda_1\ge \dots\ge \lambda_n$ are well defined elements of $\Rt$, independent
of the representative $(A_\eps)_\eps$ of $A$. By construction, $\pm\det(A-\la I) = (\la-\la_1)\dots (\la-\la_n)$,
so by Th.\ \ref{vmain} any eigenvalue of $A$ is an interleaving of $\{\la_1,\dots,\la_n\}$.
From the classical theory we know that for each $\eps$ there exists a unitary matrix $U_\eps\in M(n,\C)$
with $U_\eps^* A_\eps U_\eps = \text{diag}(\lambda_1^\eps,\dots,\lambda_n^\eps)$. Setting
$U:=[(U_\eps)_\eps]$ concludes the proof.
\end{proof}

Analogously, for skew-symmetric matrices we obtain:
\begin{prop}\label{hermevprop} Let $A=[(A_\eps)_\eps]\in M(n,\Rt)$ be skew-symmetric. Let
$\theta_k=\mu_{k,\eps} + i \nu_{k,\eps}$ ($1\le k\le n$)
be the eigenvalues of $A_\eps$, ordered by the size of their imaginary parts:
$\nu_{1,\eps}\ge \dots \ge \nu_{n,\eps}$ and set $\theta_k := [(\theta_{k,\eps})_\eps]\in\Ct$ 
($1\le i \le n$). %\blem \label{evwelldef}
Then $\theta_1,\dots,\theta_n$ are well-defined elements of $\Ct$, independent of the
representative $(A_\eps)_\eps$.
They are of the form $\pm i \lambda_k$ ($\lambda_k\in \Rt$). If $n$ is odd, then
at least one $\la_k$ equals $0$. In particular, if $A$ is non-degenerate then $n$ must be even.
Any eigenvalue of $A$ is an interleaving of $\{\theta_1,\dots,\theta_n\}$.
\end{prop}
\begin{proof} Set $\tilde A := iA$. 
Picking a skew-symmetric representative $(A_\eps)_\eps$ of $A$ as in Lemma \ref{lem:s-s representative}, we
obtain a representative $(\tilde A_\eps)_\eps$ of $\tilde A$ with all $\tilde A_\eps$
Hermitian. Denote by $\alpha_1^\eps\ge \dots \ge\alpha_n^\eps$ the eigenvalues
of $\tilde A_\eps$. By the proof of Prop.\ \ref{Hermevwelldef}, the equivalence classes
of the $(\alpha_k^\eps)_\eps$ in $\Ct$ do not depend on the representative of $\tilde A$. 
Consequently, the eigenvalues of 
$A_\eps$ (ordered by the size of their imaginary parts) define 
uniquely determined elements of $\tilde \C$, independent of the chosen 
representative of $A$.
By the classical theory, the eigenvalues of $A_\eps$ come in pairs $\pm i\lambda_k^\eps$ with 
$\lambda_k^\eps\in \R$, plus $0$ in case $n$ is odd. The final claim follows again from Th.\ 
\ref{vmain}.
\end{proof}

Based on this result, we now obtain the following normal forms of skew-symmetric generalized
matrices 

\bthm \label{lem:s-s normal forms}
 Let $A\in M(n,\Rt)$ be skew-symmetric.
 \begin{itemize}
 \item[(i)] There exists an orthogonal matrix $V\in M(n,\Rt)$ such that
 $V^tAV$ is of block-diagonal form $\mathrm{diag}(B_1,\dots,B_k,0,\dots,0)$, with each $B_j$ of the form
 \beq \nonumber \label{eq:s-s nor f}
  B_j=\left(\begin{array}{rr}
 0 & -\lambda_j \\
 \lambda_j & 0  
  \end{array}\right),
 \eeq
 followed by $1\times 1$ blocks of zeros.    Here, $\lambda_1\ge \dots \ge \lambda_k\ge 0$ and
 $\pm i \lambda_j$ ($1\le j\le k$) are the non-zero entries of the $n$-tuple constructed
 in Prop.\ \ref{hermevprop}.
 \item[(ii)] $A$ is non-degenerate if and only if each $\theta_k$ from Prop.\ \ref{hermevprop} is invertible.
 In this case there exists a non-degenerate matrix $V\in M(n,\Rt)$ such that
 \beq \nonumber\label{eq:s-s nondeg nor f}
 V^t A V=\left(\begin{array}{rr}
 0 & -I \\
 I & 0
 \end{array}\right).
 \eeq
 \end{itemize}
\ethm
\begin{proof}
(i) Let $(A_\eps)_\eps$ be a skew-symmetric representative of $A$ (Lemma \ref{lem:s-s representative}).
Then for each $\eps$ there exists an orthogonal matrix $V_\eps\in M(n,\R)$ such that 
$V_\eps^t A_\eps V_\eps$ is of block diagonal form $\mathrm{diag}(B_1^\eps,\dots,B_{k_\eps}^\eps,0,\dots,0)$, 
with each $B_j^\eps$ of the form
 \beq 
  B_j^\eps=\left(\begin{array}{rr}
 0 & -\lambda_j^\eps \\
 \lambda_j^\eps & 0  
  \end{array}\right),
 \eeq
 followed by $1\times 1$ blocks of zeros.
 In this representation, $\lambda_1^\eps\ge \dots \ge \lambda_{k_\eps}^\eps\ge 0$ are such that
 $\pm i \lambda_j^\eps$ ($1\le j\le k_\eps$) are the non-zero eigenvalues of $A_\eps$. 
 By Prop.\ \ref{hermevprop}, $\theta_1,\dots,\theta_n$ are the equivalence classes in $\Ct$ of these eigenvalues. 
 (Note that even if $\theta_j^\eps\not=0$ for all $\eps$,
 $\theta_j$ may equal $0\in \Ct$.) Since $V:=[(V_\eps)_\eps]$ is clearly orthogonal in $M(n,\Rt)$,
the statement follows.

(ii) The first claim is immediate from Lemma \ref{GKOS_Lemma}. If $A$ is invertible, then 
$A$ defines a symplectic form on $\Rt^{n^2}$, and so by Theorem \ref{basisthm} there 
exists a symplectic basis of $\Rt^n$ for $A$. It then suffices to take for $V$
the matrix with the elements of this basis as columns. 
\end{proof}

\begin{rem} An alternative, more intrinsic and general, way of isolating
a distinguished set of eigenvalues, as was done above for 
skew-symmetric and Hermitian matrices, is the following: 
Let $A\in M(n,\Ct)$ be a matrix with characteristic polynomial 
$\pm \det(A-\lambda I) = (\la-\la_1)\dots (\la-\la_n)$ with $\la_1\ge \dots \ge \la_n$. 
Since by Th.\ \ref{vmain} any eigenvalue of A is a finite interleaving of these $\la_k$, 
this determines $\la_1$ uniquely as the largest eigenvalue
of $A$. (In particular, this holds independently of the representative $(A_\eps)_\eps$ of $A$
used to construct $\la_1\ge \dots \ge \la_n$). Then by induction, all $\la_k$ are
uniquely determined: supposing that also 
$\pm \det(A-\lambda I) = (\la-\mu_1)\dots (\la-\mu_n)$ with $\mu_1\ge \dots \ge \mu_n$
then we show that
$(\la-\la_2)\dots (\la-\la_n) = (\la-\mu_2)\dots (\la-\mu_n)$,
from which by the above it will follow that $\mu_2=\la_2$ is the largest root of this polynomial,
etc.

Indeed, we know that the left hand side times $(\la-\la_1)$ equals the right hand
side times $(\la-\la_1)$ (as $\la_1=\mu_1$), so the (generalized complex) coefficients of
these polynomials are equal. This implies that the original polynomials 
are equal (in fact, if $(\la-c)(a_{n-1}\la^{n-1}+\dots+\la a_1+a_0)=0$, then 
$a_{n-1}=a_{n-2}-ca_{n-1}=\dots=a_0-ca_1=-ca_0=0$, so $a_{n-1}=\dots=a_0=0$).
\end{rem}

%%%%%%%%%%%%%%%%%%%%%%%%%%%%%%%%%%%%%%%%%%%%%%%%%%%%%%%%%%%%%%%%%%
 \subsection*{Acknowledgment}
%%%%%%%%%%%%%%%%%%%%%%%%%%%%%%%%%%%%%%%%%%%%%%%%%%%%%%%%%%%%%%%%%%

 This work was supported by projects P25236, Y-237, and P23714 of the
 Austrian Science Fund and Projects 174024, 174005 of the
 Serbian Ministry of Science, as well as project 114-451-2648/2012 
 of  the Provincial Secretariat for Science of Serbia.  
 We thank an anonymous referee for several suggestions that have substantially
 improved the presentation.

%%%%%%%%%%%%%%%%%%%%%%%%%%%%%%%%%%%%%%%%%%%%%%%%%%%%%%%%%%%%%%%%%%
%%%%%%%%%%%%%%%%%%%%%%%%%%%%%%%%%%%%%%%%%%%%%%%%%%%%%%%%%%%%%%%%%%

%%%%%%%%%%%%%%%%%%%%%%%%%%%%%%%%%%%%%%%%%%%%%%%%%%%%%%%%%%%%%%%%%%
%%%%%%%%%%%%%%%%%%%%%%%%%%%%%%%%%%%%%%%%%%%%%%%%%%%%%%%%%%%%%%%%%%
 \end{document}